\documentclass[11pt, reqno, oneside]{amsart}
\usepackage{amssymb, mathrsfs, bm, epsfig, psfrag}
\usepackage{mathabx}
\usepackage{color, xcolor}

\usepackage[utf8]{inputenc}
\usepackage[bbscaled=.95,cal=boondoxo]{mathalfa}

\usepackage{abstract}
\usepackage[all]{xy}

\usepackage[left= 3.4 cm, right= 3.4 cm, top= 2.8 cm, bottom=2.7 cm, foot= 1.2 cm, marginparwidth=2.7 cm, marginparsep=0.5 cm]{geometry}
\usepackage[inline]{enumitem}

\usepackage[bookmarks=false]{hyperref}
\hypersetup{colorlinks}
\definecolor{darkred}{rgb}{0.5,0,0}
\definecolor{darkgreen}{rgb}{0,0.5,0}
\definecolor{darkblue}{rgb}{0,0,0.5}
\hypersetup{colorlinks, linkcolor=darkblue, filecolor=darkgreen, urlcolor=darkred, citecolor=darkblue}
\makeatletter 
\@addtoreset{equation}{section}
\makeatother  

\numberwithin{equation}{section}

\newtheorem{thm}{Theorem}[section]

\newtheorem{conj}[thm]{Conjecture}
\newtheorem{prop}[thm]{Proposition}
\newtheorem{lemma}[thm]{Lemma}
\theoremstyle{definition}
\newtheorem{defn}[thm]{Definition}
\theoremstyle{remark}
\newtheorem{rem}[thm]{Remark}
\newtheorem{hyp}[thm]{Hypothesis}

\newtheorem{notation}[thm]{Notation}

\newcounter{notes}
{\end{list}}

\makeatletter
\@addtoreset{thm}{section}
\makeatother

\newcommand\qu{/\kern-.7ex/} 
\renewcommand{\setminus}{\smallsetminus}

\newcommand{\beq}{\begin{equation}}
\newcommand{\eeq}{\end{equation}}
\newcommand{\beqn}{\begin{equation*}}
\newcommand{\eeqn}{\end{equation*}}
\newcommand{\ov}{\overline}
\newcommand{\mb}{\mathbb}
\newcommand{\mc}{\mathcal}
\newcommand{\mf}{\mathfrak}

\newcommand{\wt}{\widetilde}
\newcommand{\wh}{\widehat}

\renewcommand{\i}{{\bm i}}

\newcommand{\ev}{{\rm ev}}

\newcommand{\vt}{{\rm vert}}

\title{The Symplectic Approach of Gauged Linear $\sigma$-Model}

\begin{document}

\author[Tian]{Gang Tian}
\address{
Beijing International Center for Mathematical Research\\
Beijing University\\
Beijing, 100871, China\\
and Department of Mathematics \\
Princeton University \\
Fine Hall, Washington Road \\
Princeton, NJ 08544 USA
}
\email{tian@math.princeton.edu}

\author[Xu]{Guangbo Xu}
\address{
Department of Mathematics \\
Princeton University\\
Fine Hall, Washington Road\\
Princeton, NJ 08544 USA
}
\email{guangbox@math.princeton.edu}

\date{\today}

\maketitle

\begin{abstract}
Witten's Gauged Linear $\sigma$-Model (GLSM) unifies the Gromov--Witten theory and the Landau--Ginzburg theory, and provides a global perspective on mirror symmetry. In this article, we summarize a mathematically rigorous construction of the GLSM in the geometric phase using methods from symplectic geometry. 
\end{abstract}



\setcounter{tocdepth}{1}
\tableofcontents

\section{Introduction}

Gauged linear $\sigma$-model (GLSM), introduced by Witten \cite{Witten_LGCY}, provides a fundamental framework in understanding two dimensional supersymmetric quantum field theories. It plays a significant role in the physical ``proof'' of mirror symmetry (\cite{Hori_Vafa}) and it leads to potential mathematical proof of the Landau--Ginzburg/Calabi--Yau correspondence (see the original physics literature \cite{VafaW, GVW, Martinec} and some mathematical approaches \cite{Chiodo_Ruan_2010, Clader_thesis} without using GLSM). Therefore it is an interesting and important task to develop the mathematical theory of GLSM. 

Before we outline the mathematical construction, let us mention a few previously studied mathematical theories that are related to GLSM. The first one is the Gromov--Witten theory, physically referred to as the (supersymmetric) nonlinear $\sigma$-model. It has been constructed for compact symplectic manifolds via a series of foundational work \cite{Ruan_96, Ruan_Tian, Ruan_Tian_97, Li_Tian, Fukaya_Ono}. Gromov--Witten theory has been well-studied and very influential in algebraic geometry and symplectic geometry. The second is the (orbifold) Landau--Ginzburg A-model theory, which was proposed by Witten \cite{Witten_spin} and constructed mathematically by Fan--Jarvis--Ruan \cite{FJR3, FJR2}. The third is a simple version of the GLSM (i.e., when $W = 0$ and no R-symmetry), which was introduced by Cieliebak--Gaio--Salamon \cite{Cieliebak_Gaio_Salamon_2000} and Mundet \cite{Mundet_thesis, Mundet_2003} (also see \cite{Mundet_Tian_2009, Mundet_Tian_draft}). The development of these theories provides us not only inspirations but also technical tools, especially, the virtual technique introduced in the study of Gromov--Witten theory. 

For us, a GLSM space is a quadruple $(X, G, \mu, W)$, where $X$ is a noncompact K\"aher manifold, $G$ is a reductive Lie group acting on $X$, $\mu$ is a moment map for the $G$-action, and $W: X \to {\mb C}$ is a holomorphic function (often called the superpotential) invariant under the $G$-action. We also require an ``R-symmetry'' (a ${\mb C}^*$-action on $X$) which makes $W$ quasihomogeneous. The variation of the moment map (the so-called Fayet-Iliopoulos term in physics terminology) put the theory into different phases. It is said in the ``geometric phase'' if $\mu^{-1}(0)$ intersects only with the smooth part of ${\rm Crit} W$ transversely, so that $\bar{X}_W:= {\rm Crit} W \qu G$ is an orbifold hypersurface of $\bar{X}:= X\qu G$. Among many examples, the most famous one is the quintic model, which, in the geometric phase, $\bar{X}_W$ is a quintic hypersurface inside $\mb{CP}^4$. 

Although it is not necessarily true in general, in certain cases, including the quintic model, the variation of the symplectic quotient can lead to another type of phases, namely Landau--Ginzburg models. In this phase, it is expected that the GLSM ``converges'' to the cohomological field theory of Fan--Jarvis--Ruan and Witten. This would lead to a prospective proof of the so-called Landau--Ginzburg/Calabi--Yau correspondence, and a way of computing Gromov--Witten invariants via Landau--Ginzburg models. There have been a few algebraic approaches (see \cite{Chiodo_Ruan_2010} \cite{Chiodo_Iritani_Ruan} \cite{CLLL_16}). A symplectic approach towards the LG/CY correspondence is another motivation of our project. 

The purpose of this article is to explain how to construct a so-called cohomological field theory (CohFT for short) associated to a GLSM space $(X, G, \mu, W)$. The construction in nature is similar to the construction of Gromov--Witten invariants. Namely, we start with a nonlinear elliptic equation, the gauged Witten equation, which is naturally associated with the problem. We construct virtual fundamental classes on the moduli spaces of equivalence classes of solutions to the gauged Witten equation. Using the virtual fundamental classes one can define various correlation functions (for all genera), which are multilinear functions on the (Chen-Ruan) cohomology of $\bar{X}_W$. In \cite{Tian_Xu, Tian_Xu_2, Tian_Xu_3}, the authors constructed an invariants for the so-called ``Lagrange multiplier type superpotential," based on the analysis of the gauged Witten equation over a {\it fixed} smooth Riemann surface. In this article, we describe a slightly different approach, which can deal with more general GLSM spaces (though only for the geometric phases). We can also extend the analysis to the case of degenerating families of Riemann surfaces. This allows us to define a collection of correlation functions which can be packaged into the CohFT. This article only serves as an outline of this approach, and the detail will appear in \cite{Limit1}. 

Parallel to our approach, there are also related constructions using algebra-geometric methods, see \cite{CLLL_15, CLLL_16} \cite{FJR_GLSM}. So far the algebraic-geometric methods cannot cover broad states and one has to show that correlation functions for only the narrow states also give a CohFT (for example, the space of narrow states is closed under quantum multiplication).

The remaining of this article is organized as  follows. In Section \ref{section2} we introduce the basic setting and state our results using the language of cohomological field theory. In Section \ref{section3} we explain how to write down the gauged Witten equation. Section \ref{section4} is about the moduli space of gauged Witten equations and a sketch of proof of compactness. In Section \ref{section5} we recall the abstract theory of the virtual cycle construction, which was initially brought up in \cite{Li_Tian}. In Section \ref{section6} we give a short description on how to apply the abstract virtual cycle theory to the current situation, and how to define the correlation function (the invariants) in the cohomological field theory. In Section \ref{section7} we discuss further developments and speculations. 

\subsection{Acknowledgements}

The authors would like to thank David Morrison, Edward Witten and Kentaro Hori for helpful discussions during the early stage of this project. G.X. would like to thank Chris Woodward for his interest in this project, and to thank Mauricio Romo for kindly answering many questions about the physics of GLSM.

\section{The Cohomological Field Theory of GLSM}\label{section2}

\subsection{The GLSM space}

Let $K$ be a compact Lie group, $G$ be its complexification, with Lie algebras ${\mf k}$ and ${\mf g}$. Denote $\hat{K} = K \times U(1)$, $\hat{G} = G \times {\mb C}^*$, and $\hat{\mf k} = {\mf k} \oplus {\bm i} {\mb R}$. We choose an ${\rm Ad}$-invariant metric on ${\mf k}$, so that ${\mf k}$ and $\hat{\mf k}$ are identified with their dual spaces.

Consider a noncompact K\"ahler manifold $(X, \omega_X, J_X)$ with the following structures.
\begin{enumerate}
\item A holomorphic ${\mb C}^*$-action which restricts to a Hamiltonian $S^1$-action, with a moment map $\mu_R: X \to {\rm Lie} S^1 \simeq {\bm i} {\mb R}$. This action is referred to as the {\bf R-symmetry}.	

\item A holomorphic $G$-action which restricts to a Hamiltonian $K$-action, with a moment map $\mu_K: X \to {\rm Lie} K \simeq {\mf k}$. 

\item A holomorphic function $W: X \to {\mb C}$.
\end{enumerate}
We require that these structures are compatible in the following sense.

\begin{hyp}\label{hyp21}\hfill
\begin{enumerate}
\item \label{hyp21a} The R-symmetry commutes with the $G$-action.

\item \label{hyp21b} $W$ is homogeneous of degree $r$, i.e., there is a positive integer $r$ such that 
\beqn
W(\xi x) = \xi^r W(x),\ \forall \xi \in {\mb C}^*,\ x\in X.
\eeqn

\item \label{hyp21c} $W$ is invariant with respect to the $G$-action.

\item \label{hyp21d} $0\in {\mf k}$ is a regular value of $\mu_K$.

\item \label{hyp21e} $\mu_K^{-1}(0)$ intersects transversely with only the smooth part of ${\rm Crit} W$. (Here ${\rm Crit} W$ is an analytic subvariety of $X$.)

\item \label{hyp21f} Let $X_W$ be the union of all irreducible components of ${\rm Crit} W$ that have nonempty intersections with $\mu_K^{-1}(0)$. Then $\mu_K|_{X_W}$ is proper, and there is a homomorphism $\iota_W: {\mb C}^* \to G$ such that 
\beqn
\xi \cdot x = \iota_W (\xi) \cdot x,\ \forall x\in X_W.
\eeqn

\end{enumerate}
\end{hyp}

\begin{rem}
\begin{enumerate}

\item Item \eqref{hyp21a} of Hypothesis \ref{hyp21} implies that there is a well-defined $\hat{G}$-action on $X$ with a moment map $\hat\mu =  (\mu_R, \mu_K)$ for the $\hat{K}$-action.

\item The objects considered in this paper are certain ``fields'' on surfaces with cylindrical ends that are asymptotic to $X_W \cap \mu_K^{-1}(0)$. On the cylindrical ends the fields satisfy roughly a gradient flow equation, and the transverse intersection $X_W \cap \mu_K^{-1}(0)$ means an Morse-Bott type asymptotic constrain.

\item The above hypothesis also allows us to consider the symplectic reductions
\begin{align*}
&\ \bar{X}:= \mu_K^{-1}(0)/K, &\ \bar{X}_W:= [ X_W \cap \mu_K^{-1}(0) ]/ K.
\end{align*}
Here $\bar{X}$ is a compact K\"ahler orbifold and $\bar{X}_W \subset \bar{X}$ is a closed suborbifold. 

\end{enumerate}
\end{rem}

For $\xi = (\xi_R, \xi_K) \in \hat{\mf k}$, denote by $\xi_R \mapsto {\mc X}_{\xi_R}$ and $\xi_K \mapsto {\mc Y}_{\xi_K}$ the infinitesimal R-symmetry and the infinitesimal $K$-actions respectively and denote ${\mc X}_\xi = {\mc X}_{\xi_R} + {\mc Y}_{\xi_K}$.

The following shows an important property of $W$.
\begin{lemma}\label{lemma23}
Any critical value of $W$ must be zero.
\end{lemma}

\begin{proof}
By the homogeneity of $W$ with respect to the R-symmetry, at a critical point $x$ of $W$, we have $0 = {\mc X}_{\xi_R} W(x) = r \xi_R  W(x)$ for all $\xi_R$. Hence $W(x) = 0$.
\end{proof}

We also have the following assumption on the geometry at infinity. This is sometimes called a convex structure (cf. \cite[Section 2.5]{Cieliebak_Gaio_Mundet_Salamon_2002}).

\begin{hyp}\label{hyp24}
There exists $\xi_W$ in the center $Z({\mf k})\subset {\mf k}$, and a continuous function $\tau \mapsto c_W (\tau)$ (for $\tau \in Z({\mf k})$) satisfying the following condition. If we define ${\mc F}_W:= \mu_K\cdot \xi_W$, then ${\mc F}_W$ is proper and
\beq
\begin{array}{c} x \in X_W, \xi \in T_x X\\
  {\mc F}_W(x) \geq c_W(\tau)
                 							
\end{array} \Longrightarrow \left\{ \begin{array}{c} \langle \nabla_\xi \nabla {\mc F}_W(x), \xi\rangle + \langle \nabla_{J \xi} \nabla {\mc F}_W(x), J \xi \rangle \geq 0, \\
\langle \nabla {\mc F}_W(x), J {\mc Y}_{\mu_K(x) - \tau}(x) \rangle \geq 0. \end{array}\right. 
\eeq
\end{hyp}

\subsection{Example}

We first look at the example in which $\bar{X}_W$ gives hypersurfaces in weighted projective spaces. Consider a quasihomogeneous polynomial $Q: {\mb C}^N \to {\mb C}$. Namely, there are integers $r_1, \ldots, r_N, r$ such that 
\beqn
Q(\xi^{r_1} x_1, \ldots, \xi^{r_N} x_N) = \xi^r Q(x_1, \ldots, x_N),\ \forall \xi \in {\mb C}^*, (x_1, \ldots, x_N) \in {\mb C}^N.
\eeqn
It is called nondegenerate if $0\in {\mb C}^N$ is the only critical point of $Q$. It induces a ${\mb C}^*$-action (the $R$-symmetry) on ${\mb C}^N$ by 
\beqn
\xi \cdot (x_1, \ldots, x_N) = (\xi^{r_1} x_1, \ldots, \xi^{r_N} x_N).
\eeqn
The corresponding weighted projective space is ${\mb P}^{N-1}(r_1, \ldots, r_N) = ({\mb C}^N \setminus \{0\})/ {\mb C}^*$. $Q$ then defines a weighted projective hypersurface $\bar{X}_Q \subset {\mb P}^{N-1}(r_1, \ldots, r_N)$. 

Following Witten \cite{Witten_LGCY}, introduce $W: {\mb C}^{N+1}\to {\mb C}$ which is defined as
\beqn
W(p, x_1, \ldots, x_N) = p Q(x_1, \ldots, x_N). 
\eeqn
The critical locus of $W$ decomposes as 
\beq\label{eqn22}
{\rm Crit} W = \Big\{ (0, x_1, \ldots, x_N)\ |\ Q(x_1, \ldots, x_N) = 0 \Big\} \cup \Big\{ (p, 0, \ldots, 0) \ |\ p \in {\mb C} \Big\}.
\eeq
Consider another group action by $G= {\mb C}^*$ on ${\mb C}^{N+1}$ given by 
\beqn
\rho \cdot (p, x_1, \ldots, x_N) = (\rho^{-r} p, \rho^{r_1} x_1, \ldots, \rho^{r_N} x_N).
\eeqn
It is Hamiltonian and a moment map is 
\beqn
\mu_K (p, x_1, \ldots, x_N) = - \frac{\i}{2} \Big[ r_1 |x_1|^2 + \cdots + r_N |x_N|^2 - r |p|^2 - \tau \Big]
\eeqn
where $\tau$ is a constant, playing the role as a parameter of this theory. 

If we take $\tau>0$, then only the first component of \eqref{eqn22} intersects with $\mu_K^{-1}(0)$. One sees that 
\beqn
\bar{X}_W = \bar{X}_Q. 
\eeqn
Item \eqref{hyp21f} of Hypothesis \ref{hyp21} holds for $\iota_W (\xi) = \xi \in {\mb C}^* \simeq  G$, since along $X_W$, the coordinate $p$ vanishes. Moreover, for Hypothesis \ref{hyp24}, ${\mc F}_W$ is roughly $r_1|x_1|^2 + \cdots + r_N |x_N|^2 - r|p|^2$. 

When $\tau<0$, the theory is related to the Landau--Ginzburg theory. This is not a geometric phase since item \eqref{hyp21e} of Hypothesis \ref{hyp21} fails. 

In \cite{Witten_LGCY} the above example were extended the cases in which $\bar{X}_W$ gives complete intersections in toric varieties. There are a few more examples that appeared more recently in physics literature (see \cite{JKLMR_2, JKLMR}) where $G$ is nonabelian. 

\subsection{The cohomological field theory}

The notion of cohomological field theory (see \cite{Manin_2}) is a way to axiomize Gromov--Witten theory and other similar theories. We recall its definition. Let $\Lambda$ be a commutative algebra over ${\mb Q}$, and let $\ov{\mc M}_{g, k}$ be the Deligne--Mumford space of stable curves.

\begin{defn}\cite{Manin_2}\label{defn25}
A $\Lambda$-valued cohomological field theory (CohFT for short) consists of a ${\mb Z}_2$-graded $\Lambda$-vector space ${\mc H}$ (called the state space) equipped with a $\Lambda$-valued nondegenerate bilinear form, and a collection of correlators ($\Lambda$-multilinear functions)
\beqn
\langle \sqbullet \rangle_{g, k}: {\mc H}^{\otimes k} \times H^*( \ov{\mc M}_{g, k}; {\mb Q}) \to \Lambda,\ g, k \geq 0,
\eeqn
whose evaluation on homogeneous elements is denoted by $\langle \alpha_1 \otimes \cdots \otimes \alpha_k; \beta \rangle_{g, k}$, that satisfy the following two splitting axioms.

\begin{enumerate}
\item {\bf Non-separating node.} Suppose $2g+ k> 3$. Let $\gamma\in H^2(\ov{\mc M}_{g, k};{\mb Q})$ be the class dual to the divisor of configurations obtained by shrinking a non-separating loop (which is the image of a map $\iota_\gamma: \ov{\mc M}_{g-1, k+2} \to \ov{\mc M}_{g, k}$). Then 
\beqn
\langle \alpha_1\otimes \cdots \otimes \alpha_k; \beta \cup \gamma \rangle_{g, k} = \langle \alpha_1 \otimes \cdots \otimes \alpha_k \otimes \Delta; \iota_\gamma^* \beta \rangle_{g-1, k+2}.
\eeqn
Here $\Delta = \sum_j \delta_j \otimes \delta^j \in {\mc H} \otimes {\mc H}$ is the ``diagonal'' class representing the pairing.

\item {\bf Separating node.} Let $\gamma \in H^2(\ov{\mc M}_{g, k}; {\mb Q})$ be the class dual to the divisor of configurations obtained by shrinking a separating loop, which is the image of a map $\iota_\gamma: \ov{\mc M}_{g_1, k_1 + 1} \times \ov{\mc M}_{g_2, k_2 + 1} \to \ov{\mc M}_{g, k}$, also characterized by a decomposition $\{1, \ldots, k\} = I_1 \sqcup I_2$ with $|I_i| = k_i$. Suppose $2g_i - 2 + k_i \geq 0$. Then for any $\beta \in H^*(\ov{\mc M}_{g, k};{\mb Q})$, if we write $\iota_\gamma^* \beta = \sum_l \beta_l' \otimes \beta_l''$ by K\"unneth decomposition, then 
\beqn
\langle \alpha_1 \otimes \cdots \otimes \alpha_k; \beta \cup \gamma \rangle_{g, k} = \epsilon(I_1, I_2) \sum_{j, l} \langle \alpha_{I_1} \otimes \delta_j; \beta_l' \rangle_{g_1, k_1+1} \langle \alpha_{I_2}\otimes \delta^j; \beta_l'' \rangle_{g_2, k_2+1}\eeqn
where $\epsilon(I_1, I_2)$ is the sign of permutations for odd-dimensional $\alpha_i$'s.
\end{enumerate}
\end{defn}

Recall the CohFT associated to Gromov--Witten theory. Let $\Lambda$ be the Novikov ring
\beqn
\Lambda:= \Big\{ \sum a_i q^{\lambda_i}\ |\ a_i \in {\mb Q},\ \lambda_i \in {\mb R},\ \lim_{i \to \infty} \lambda_i = +\infty \Big\}.
\eeqn
In Gromov--Witten theory, for a compact symplectic manifold $(M, \omega)$, ${\mc H} = H^*(M; \Lambda)$ is equipped with the Poincar\'e pairing. The correlation functions $\langle \alpha_1 \otimes \cdots \alpha_k; \beta \rangle_{g, k}$ are just the Gromov--Witten invariants. It is defined as follows. Taking a compatible almost complex structure $J$ on $M$, the moduli space of stable $J$-holomorphic curves representing class $A \in H_2(M; {\mb Z})$, denoted by $\ov{\mc M}_{g, k}(M; A)$, admits a virtual fundamental cycle. There is the evaluation map 
\beqn
{\rm ev}: \ov{\mc M}_{g, k}(M; A) \to M^k \times \ov{\mc M}_{g, k}.
\eeqn
The virtual fundamental cycle is pushed forward by the evaluation map to a class
\beqn
\big[ \ov{\mc M}_{g, k}(M; A) \big]^{\rm vir} \in H_*(M^k \times \ov{\mc M}_{g, k}; {\mb Q}).
\eeqn
The pairing against the cohomology class $\alpha_1 \otimes \cdots \otimes \alpha_k$ and $\beta$ gives the correlation function
\beqn
\langle \alpha_1 \otimes \cdots \otimes \alpha_k; \beta \rangle_{g, k}:= \sum_{A \in H_2(M; {\mb Z})} q^{\omega(A)} \Big\langle \big[ \ov{\mc M}_{g, k}(M; A) \big]^{\rm vir},  \alpha_1 \otimes \cdots \otimes \alpha_k \otimes \beta \Big\rangle.
\eeqn
The fact that this is a well-defined $\Lambda$-multilinear map follows from Gromov compactness. The two splitting axioms of CohFT can be derived from the properties of the virtual fundamental cycles (see \cite{Ruan_Tian_97} and \cite{Fukaya_Ono}). Further, the Gromov--Witten theory were extended to orbifolds in \cite{Chen_Ruan_2002}, where the state spaces are replaced by the Chen-Ruan orbifold cohomology \cite{Chen_Ruan_cohomology}. 

The Landau--Ginzburg correlation functions defined in \cite{FJR2} can also be packaged into a CohFT. In that case, the moduli spaces do not have components labelled by the degree (namely, the class $A$) so the use of the Novikov ring is not necessary.

We give a formal description of the CohFT associated to a GLSM space $(X, G, W, \mu)$, whose construction occupies the rest of this paper. The state space is ${\mc H}:= H_{\rm CR}^*(\bar{X}_W; \Lambda)$, the Chen-Ruan orbifold cohomology of $\bar{X}_W$. When the quotient $\bar{X}_W= [X_W \cap \mu_K^{-1}(0)] /K$ is free (for example, in the quintic model), $H_{\rm CR}^*(\bar{X}_W; \Lambda)$ is the same as the ordinary cohomology. When the quotient is not free, there exist finite order elements $\hat\gamma \in K$ such that $X_{W, \hat\gamma}\cap \mu_K^{-1}(0) = [{\rm Fix} (\hat\gamma) \cap X_W ]\cap \mu_K^{-1}(0) \neq \emptyset$. The quotient $[X_{W, \hat\gamma} \cap \mu_K^{-1}(0)]/K$ is denoted by $\bar{X}_{W, [\hat\gamma]}$, which only depends on the conjugacy class $[\hat\gamma]$. In this case, $H_{\rm CR}^*(\bar{X}_W, \Lambda)$ is the direct sum of the ordinary cohomology of $\bar{X}_W$ with components corresponding to the ``twisted sectors,'' namely, $H^*( \bar{X}_{W; [\hat\gamma]}; \Lambda)$. The degree in the twisted sectors needs to be shifted by certain rational numbers $\iota([\hat\gamma])$, i.e., for $\alpha \in H^*(\bar{X}_{W, [\hat\gamma]}; \Lambda)$, 
\beqn
{\rm deg}_{\rm CR}(\alpha) = {\rm deg}(\alpha) - \iota([\hat\gamma]).
\eeqn
The ${\mb Z}_2$-grading on ${\mc H}$, though, is still defined by the ordinary cohomology degree.

To define the correlation functions, one consider certain moduli spaces $\ov{\mc M}_{g, k}(X, W; B)$ of gauge equivalence classes of solutions to the {\bf gauged Witten equation}. Here $B$ is an equivariant class in $H_2^K({X}_W; {\mb Q})$, playing a similar role as the class $A\in H_2(M; {\mb Z})$ in Gromov--Witten theory. In particular, $B$ can be evaluated against the equivariant Chern class $c_{\, 1}^K(TX)$ and the equivariant symplectic class $\omega_K$. We will describe the gauged Witten equation in the next section. There is an evaluation map 
\beqn
\ev: \ov{\mc M}_{g, k}(X, W; B) \to \bigsqcup_{[\hat\gamma]} \bar{X}_{W,[\hat\gamma]} \times \ov{\mc M}_{g,k}.
\eeqn
Notice that the ordinary cohomology of the disjoint union $\bigsqcup_{[\hat\gamma]} \bar{X}_{W, [\hat\gamma]}$ (the so-called inertia stack) is the Chen-Ruan cohomology. Moreover, one can construct a virtual fundamental cycle on $\ov{\mc M}_{g, k}(X, W; B)$, which is pushed forward via the above evaluation map to a class
\beqn
\big[ \ov{\mc M}_{g, k}(X, W; B)\big]^{\rm vir} \in \left[ \bigoplus_{[\hat\gamma]} H_* (\bar{X}_{W, [\hat\gamma]}; {\mb Q}) \right]^{\otimes k} \otimes H_*(\ov{\mc M}_{g, k}; {\mb Q}).
\eeqn
This class can be evaluated against cohomology classes. Our main result is 
\begin{thm}\cite{Limit1}
The collection of virtual cycles $[\ov{\mc M}_{g, k}(X, W; B)]^{\rm vir}$ define (in the same fashion as defining orbifold Gromov--Witten invariants) a CohFT on $H_{\rm CR}^*(\bar{X}_W; \Lambda)$.
\end{thm}

\section{The Gauged Witten Equation}\label{section3}

In this section we introduce the gauged Witten equation. It can be viewed as a combination of the symplectic vortex equation (\cite{Cieliebak_Gaio_Salamon_2000}, \cite{Mundet_thesis, Mundet_2003}) and the Witten equation (\cite{Witten_spin}, \cite{FJR1, FJR2, FJR3}). In particular, compared to Gromov--Witten theory, there are two additional features. First, the equation is defined not for ordinary curves but for $r$-spin curves, the same for the Witten equation. Second, the equation is not conformal invariant and we need to choose volume forms on curves. The second feature is actually very important. Because of the flexibility of volume forms, we can manipulate them for our purposes.

\subsection{$r$-spin curves}

\begin{defn}
A (smooth) $r$-spin curve ${\mc C}$ consists of a compact orbifold Riemann surface $\Sigma_{\mc C}$ whose orbifold points are $z_1, \ldots, z_k$, an orbifold holomorphic line bundle $L_R \to \Sigma_{\mc C}$, and an isomorphism 
\beq\label{eqn31}
\varphi_R: L_R^{\otimes r} \simeq K_{\Sigma_{\mc C}} \otimes {\mc O}(z_1) \otimes \cdots \otimes {\mc O}(z_k).
\eeq
The orbifold structure of $L_R$ at $z_j$ is classified by an integer $m_j \in \{0, 1, \ldots, r-1\}$. If $m_j = 0$, we say that $z_j$ is {\bf broad}; otherwise it is called {\bf narrow}.
\end{defn}

The right hand side of \eqref{eqn31} is called the log-canonical bundle of ${\mc C}$, denoted by $\omega_{\mc C}$. 

The moduli space of stable $r$-spin curves of genus $g$ with $k$ orbifold points, denoted by $\ov{\mc M}{}_{g, k}^r$, is a branched cover over the Deligne-Mumford space $\ov{\mc M}_{g, k}$ (see \cite{FJR2} for more details about $\ov{\mc M}{}_{g, k}^r$, or the moduli space of more general $W$-curves). For any collection $m_1, \ldots, m_k\in \{0, 1, \ldots, r-1\}$, we have a connected component (could be empty) $\ov{\mc M}{}_{g, k}^r(m_1, \ldots, m_k) \subset \ov{\mc M}{}_{g, k}^r$.

Let $\Sigma_{\mc C}^*$ be the (open) surface obtained by removing the punctures $z_1, \ldots, z_k$. There is a way of choosing a cylindrical metric on $\Sigma_{\mc C}^*$ which varies smoothly over $\ov{\mc M}{}_{g, k}^r$. Further, there is a canonically induced Hermitian metric $H_R$ on $L_R|_{\Sigma_{\mc C}^*}$ from the cylindrical metric on $\Sigma_{\mc C}^*$, using the $r$-spin structure. Let $A_R$ be the Chern connection of $H_R$, which varies smoothly over $\ov{\mc M}{}_{g, k}^r$. A useful fact is that the curvature form $F_{A_R}$ of $A_R$ is uniformly bounded with respect to the cylindrical metric.

Let $P_R\to \Sigma_{\mc C}^*$ be the unit circle bundle with respect to the metric $H_R$, which is a principal $S^1$-bundle. Consider a principal $K$-bundle $P_K \to \Sigma_{\mc C}^*$ and let $P_{\mc C}^*$ be the fibre product $P_R \times_{\Sigma_{\mc C}^*} P_K$, which is a principal $\hat{K}$-bundle. Form the associated fibre bundle $\pi_Y: Y_{\mc C}^* \to \Sigma_{\mc C}^*$, where $Y_{\mc C}^*:= P_{\mc C}^* \times_{\hat K} X$. We lift $W$ to ${\mc W}\in \Gamma( Y_{\mc C}^*, \pi_Y^* \omega_{\mc C})$, defined by 
\beqn
{\mc W}([p_R, p_K, x]) = W(x) \varphi_R( (p_R)^r).
\eeqn
By the fact that $W$ is homogeneous of degree $r$, ${\mc W}$ is a well-defined section. 

Consider arbitrary connection 1-forms $A_K \in {\mc A}(P_K)$. Let $A = (A_R, A_K)$ be the induced connection on $P_{\mc C}^*$. We denote by ${\mc A}_W (P_{\mc C}^*)\subset {\mc A}(P_{\mc C}^*)$ the subset of connections obtained in this way. An important fact is that each $A$ induces an integrable complex structure on the total space $Y_{\mc C}^*$ and ${\mc W}$ is holomorphic with respect to this complex structure.

The K\"ahler metric is $\hat K$-invariant, hence there is an induced Hermitian metric on the vertical tangent bundle $T^\vt Y_{\mc C}^* \to Y_{\mc C}^*$. Therefore, one can dualize the differential $d{\mc W}^\vt \in \Gamma(Y_{\mc C}^*, \pi_Y^*\omega_{\mc C} \otimes (T^\vt Y_{\mc C}^* )^\vee)$, obtaining the gradient
\beqn
\nabla {\mc W} \in \Gamma(Y_{\mc C}^*, \pi_Y^* \omega_{\mc C}^\vee \otimes T^\vt Y_{\mc C}^* ).
\eeqn
On the other hand, for each $A \in {\mc A}_W(P_{\mc C}^*)$ and $u \in {\mc S}(Y_{\mc C}^*)$, one can form 
\beqn
\ov\partial_A u \in \Gamma( \Sigma_{\mc C}^*, \omega_{\mc C} \otimes u^* T^\vt Y_{\mc C}^*)=: \Omega^{0,1}(\Sigma_{\mc C}^*, u^* T^\vt Y_{\mc C}^*).
\eeqn
It lies in the same space as $u^* \nabla {\mc W} =: \nabla {\mc W} (u)$. Hence we can write the {\it Witten equation}
\beqn
\ov\partial_A u + \nabla {\mc W} (u) = 0.
\eeqn	
We need to fixed the complex gauge by imposing a curvature condition on $A_K$, that is, the following {\it vortex equation}
\beq\label{eqn32}
* F_{A_K} + \mu_K(u) = 0,
\eeq
where $*$ is the Hodge star operator with respect to the cylindrical metric. The {\bf gauged Witten equation} is the following system on $(A, u)\in {\mc A}_W (P_{\mc C}^*) \times {\mc S}(Y_{\mc C}^*)$
\begin{align}\label{eqn33}
&\ \ov\partial_A u + \nabla {\mc W} (u) = 0,\ &\ * F_{A_K} + \mu_K (u) = 0.
\end{align}
There is an obvious gauge symmetry of this system, similar to the case of symplectic vortex equation. However, one has a crucial difference, which is, the first equation is not invariant under complex gauge transformation.

\subsection{Energy}

Choose $p>2$. Consider pairs $(A, u) \in {\mc A}^{1, p}_{W; {\rm loc}}(P_{\mc C}^*) \times {\mc S}^{1, p}_{\rm loc}(Y_{\mc C}^*)$. With respect to a local trivialization of $P_{\mc C}^*$ over a chart $U \subset \Sigma_{{\mc C}}^*$, $u$ can be viewed as a map from $U$ to $X$ and $A$ is identified with $d + \phi ds + \psi dt$, where $\phi, \psi\in W^{1, p}_{\rm loc}(U, \hat{\mf k})$ and $s, t$ are local coordinates of $U$. The covariant derivative of $u$, in local coordinates, reads
\beqn
d_A u = ds \otimes (\partial_s u + {\mc X}_\phi(u)) + dt \otimes (\partial_t u + {\mc X}_\psi(u)).
\eeqn
For such a pair $(A, u)$, define its {\bf energy} as
\begin{multline*}
E(A, u) =  \frac{1}{2} \Big( \| d_A u \|_{L^2}^2 + \| \mu_K (u) \|_{L^2}^2 + \| F_{A_K} \|_{L^2}^2 \Big) + \| \nabla {\mc W} (u)\|_{L^2}^2\\
        = \frac{1}{2} \left[ \int_{\Sigma_{{\mc C}}^*} \Big( | d_A u|^2 +  |\mu_K(u)|^2+ |F_{A_K}|^2 + 2 |\nabla {\mc W} (u)|^2  \Big) \sigma_c \right].
\end{multline*}
A solution over $C \subset \Sigma_{\mc C}^*$ is called {\bf bounded} if it has finite energy and there exists a $\hat{K}$-invariant compact subset $N \subset X$ such that $u(C) \subset P\times_{\hat{K}} N$. 

Let $\wt{\mc M}_{\mc C}(X, W)$ be the set of smooth bounded solutions and ${\mc M}_{\mc C}(X, W)$ be its quotient by smooth gauge transformations.

\subsection{Asymptotic behavior and the homology class of solutions}

We show the following properties of the gauged Witten equation. Denote $Y_W^*:= P_{\mc C}^* \times_{\hat K} X_W$. 

\begin{thm}\label{thm33}
Given a smooth $r$-spin curve ${\mc C}$ and ${\bf v} = (A, u)\in \wt{\mc M}_{\mc C}(X, W)$. Then
\begin{enumerate}

\item \label{thm33a} $u$ is holomorphic with respect to $A$ and $u(\Sigma_{\mc C}^*) \subset Y_W^*$.

\item \label{thm33b} For each puncture $z_j$ of ${\mc C}$ whose $r$-spin monodromy is $\exp(2\pi m_j/r) \in {\mb Z}_r$, there exist $x_j \in \mu_K^{-1}(0)\cap X_W$ and $\lambda_j \in {\mf k}$, such that under certain trivialization of $Y_{\mc C}^*$, 
\beqn
\lim_{z \to z_j} e^{-[ {\bm i} m_j/r +  \lambda_j] t} u(s, t) = x_j,\ e^{2\pi {\bm i} m_j/r} e^{2\pi \lambda_j} \cdot x_j = x_j.
\eeqn
\end{enumerate}
\end{thm}

\begin{proof}[Sketch of Proof]
Firstly, using similar method as 	\cite[Proposition 4.5]{Tian_Xu}, one can prove that $\nabla{\mc W} (u)$ converges to zero as approaching to punctures. By a simple computation and using Stokes' theorem, $\| \ov\partial_A u\|_{L^2}^2$ is equal to the sum of limits of ${\mc W} (u)$ at punctures. It follows from Lemma \ref{lemma23} that $\ov\partial_A u = 0$ and hence $u(\Sigma_{\mc C}^*) \subset Y_W^*$, hence prove Item \eqref{thm33a}. Next, one can show that $\mu_K(u)$ converges to zero as approaching to punctures. Therefore, near punctures, one can project $u$ to a holomorphic map into $\bar{X}_W$ (here the projection is holomorphic by the K\"ahler condition). By the usual removal of singularity (for orbifold holomorphic maps), one obtains Item \eqref{thm33b}.
\end{proof}

Denote $\hat\gamma_j = e^{2\pi [\iota_W({\bm i} m_j/r) + \lambda_j]}\in K$. Its conjugacy class $[\hat\gamma_j]$ in $K$ is gauge-invariant. Then $x_j \in {\rm Fix} X_{\hat\gamma_j} \cap \mu_K^{-1}(0) \cap X_W$. Since the $K$-action on $\mu_K^{-1}(0)$ has only finite stabilizers, $[\hat\gamma_j] $ belongs to only a finite collection of possible conjugacy classes.  We define an orbifold completion $Y_{\mc C}([\hat\gamma_1], \ldots, [\hat\gamma_k]) \to \Sigma_{\mc C}$ of $Y_{\mc C}^*$, constructed by adding orbifold charts near $z_j$, whose transition function between the original chart and the orbifold chart is $e^{[\iota_W({\bm i} m_j/r) + \lambda_j] t}$. Then Theorem \ref{thm33} implies that $u$ extends to a continuous orbifold section of $Y_{\mc C}([\hat\gamma_1], \ldots, [\hat\gamma_k])$. Therefore ${\bf v}$ defines a rational equivariant homology 2-class of $X_W$, denoted by $[{\bf v}]$. Its image in $H_2^K(X; {\mb Q})$ is also denoted by $[ {\bf v} ]$. Notice that $[{\bf v}]$ contains the topological information of the bundles $L_R$ and $P_K$. One can prove an {\it a priori} energy bound using similar computations as in the case of symplectic vortex equation (see for example \cite{Cieliebak_Gaio_Mundet_Salamon_2002}).

\begin{thm}\label{thm34} There is a function $e: H_2^K(X; {\mb Q}) \to {\mb R}$ satisfying the following condition. Given a bounded solution ${\bf v}$ to the gauged Witten equation over ${\mc C}$ satisfying the conditions of Theorem \ref{thm33}, one has $E({\bf v}) \leq e([{\bf v}])$.
\end{thm}

\begin{rem}\label{rem35}
The energy bound is usually necessary to prove compactness. In the theory of vortex equation, one has the energy identity, i.e., the energy is exactly equal to a topological quantity. In our case, we can achieve that by varying the connection $A_R$ and requiring $A_R$ to satisfy the vortex equation. This is actually the choice made in \cite{Tian_Xu}. 
\end{rem}

\section{Moduli Space and Compactness}\label{section4}

The moduli space needs to be compactified. Similar to the case of Gromov--Witten theory, there are a few issues we need to take care of. The first one is the bubbling. Since we have assume that $X$ is aspherical, sphere bubbling cannot happen. The second one is to study the limiting behavior of solutions on a sequence of $r$-spin curves that degenerate to a nodal curve. This is a bit subtle when we have gauge fields. For example, in \cite{Mundet_Tian_2009} the compactification of moduli space of symplectic vortices is much more complicated. However, this is partially caused by using the smooth metric for the vortex equation. Our choice of cylindrical metric for the vortex equation can eliminate this complexity and simplify the compactification. Thirdly, at punctures or when nodes are forming, a positive amount of energy can stay on an infinite cylinder and form solitons. Lastly, since the target space is noncompact, one has to prove a uniform $C^0$ bound. The $C^0$ bound is derived using an argument of \cite{Cieliebak_Gaio_Mundet_Salamon_2002}, using Hypothesis \ref{hyp24} and the holomorphicity of solutions. 

In this section, we first define the notion of solitons. Then we will introduce the notion of stable solutions. Lastly we give a rather complete proof of the $C^0$ bound, which is the crucial part of the proof of compactness. 

\subsection{Solitons}

\begin{defn}
Let $\Theta = {\mb R} \times S^1$ be the infinite cylinder. Given $m \in \{0, 1, \ldots, r-1\}$ and denote $\lambda_R = {\bm i} m /r$. A soliton is a triple ${\bf v}= (u, \phi, \psi)$, where $u \in W^{1, p}_{\rm loc}(\Theta, X_W )$ is a map, $\phi, \psi \in W^{1, p}_{\rm loc}(\Theta, {\mf k})$, such that they solves the equation
\beq\label{eqn41}
\left\{ \begin{array}{r} \partial_s u + {\mc Y}_{\phi} + J (\partial_t u + {\mc X}_{\lambda_R} + {\mc Y}_{\phi}) = 0,\\
													 \partial_s \psi - \partial_t \phi + [\phi, \psi] + \mu_K(u) = 0.
\end{array}
\right. 
\eeq
\end{defn}

One can define a notion of energy for solitons. It is easy to prove that for any bounded solution to \eqref{eqn41}, one can gauge transform it to a solution ${\bf v} = (u, \phi, \psi)$ such that 
\begin{enumerate}
\item There exists $\lambda_\pm \in {\mf k}$ such that $\displaystyle \lim_{s \to \pm \infty} (\phi(s, t), \psi(s, t)) = (0, \lambda_\pm)$.

\item There exists $x_\pm \in X_W \cap \mu_K^{-1}(0)$ such that
\begin{align}\label{eqn42}
&\ \lim_{s \to \pm \infty} e^{- [\lambda_R + \lambda_\pm] t} u(s, t) = x_\pm,\ &\ e^{2\pi (\lambda_R + \lambda_\pm)}\cdot x_\pm = x_\pm.
\end{align}
\end{enumerate}
It follows that a soliton ${\bf v}$ represents a rational equivariant homology class in $H_2^K( X_W; {\mb Q})$. To establish compactness for solitons, one also needs to prove the energy quantization property, namely, there is a positive lower bound for energy of solitons. We skip the details here.

\subsection{Stable solutions}

Let us fix genus $g$ and the number of marked points $k$. The domain of a stable soliton is a prestable $r$-spin curve $\wt{\mc C}$ of genus $g$ with $k$ marked points, such that its stabilization ${\mc C}$ is obtained from $\wt{\mc C}$ by contracting all unstable rational components with two special points, each of which is the domain of a nontrivial soliton. Such a domain can be described by a dual graph $\Gamma$. Moreover, on the principal components $\Sigma_\alpha$ (i.e., those that are not contracted by stabilization) one has chosen a volume form which is of cylindrical type and which varies smoothly over $\ov{\mc M}{}_{g, k}^r$.

\begin{defn}
Let $\wt{\mc C}$ be a prestable $r$-spin curve with combinatorial type $\Gamma$. A {\bf stable solution} with underlying curve $\wt{\mc C}$ consists of a collection of smooth bounded solutions ${\bf v}_\alpha$ to the gauged Witten equation for all principal components $\Sigma_\alpha$ and a collection of solitons ${\bf v}_\delta$ for all bubble components $\Sigma_\delta$. These solutions should satisfy the following matching conditions at nodes. Namely, for each node $w$ with preimages $w', w''$ in the normalization of $\wt{\mc C}$, the monodromies of the solutions around $w'$, $w''$, denoted by $\hat\gamma', \hat\gamma'' \in K$, should satisfy $\hat\gamma' \hat\gamma'' = 1$. Moreover, the evaluations at $w', w''$ in $\bar{X}_{W, [\hat\gamma']} = \bar{X}_{W, [\hat\gamma'']}$ coincide. 

The homology class of a stable solution is the sum in $H_2^K(X_W; {\mb Q})$ of the classes of all components. For each dual graph $\Gamma$ and $B \in H_2^K(X_W; {\mb Q})$, let ${\mc M}_\Gamma(X, W; B)$ be the moduli space of gauge equivalence classes of stable solutions over prestable $r$-spin curves $\wt{\mc C}$ that has the combinatorial type $\Gamma$ and homology class $B$. 
\end{defn}

 One can define a partial order $\Gamma' \leq \Gamma$ among dual graphs. Define $\ov{\mc M}_{\Gamma}(X, W; B):= \displaystyle \sqcup_{\Gamma' \leq \Gamma} {\mc M}_{\Gamma}(X, W; B)$. One can define the topology of $\ov{\mc M}_\Gamma(X, W; B)$ in the usual way for gauged maps. Although solutions to the gauged Witten equations are not typical vortices, the topology only depends how we define the convergence locally. We call this topology the {\bf G-topology}. Moreover, it is also a standard procedure to prove the following fact, thanks to the uniform $C^0$ bound.

\begin{thm}\label{thm43}
For a dual graph $\Gamma$ representing the combinatorial type of a genus $g$ $r$-spin curve with $k$ punctures, stable up to soliton components, and a homology class $B \in H_2^K(X_W; {\mb Q})$, $\ov{\mc M}_\Gamma(X, W; B)$ is compact and Hausdorff. 
\end{thm}

\subsection{Compactness}

The compactness part of Theorem \ref{thm43} relies crucially on a uniform $C^0$ bound on solutions. In the old setting of \cite{Tian_Xu}, because we need to perturb the equation, the proof becomes very complicated due to the perturbation term and the loss of the holomorphicity of solutions. In the current setting the $C^0$ bound can be derived in a way similar to the case of \cite{Cieliebak_Gaio_Mundet_Salamon_2002}, essentially by the convexity and the properness of the moment map. The proof is now extended to the situation when the conformal structure of the curves degenerate.

In more precise terms, the result on uniform $C^0$ bound is stated as follows.
\begin{prop}
For any equivariant curve class $B\in H_2^K(X_W; {\mb Q})$, there exists a $\hat K$-invariant compact subset $N = N(B) \subset X$ satisfying the following conditions. Given a smooth $r$-spin curve ${\mc C}$ and $(A, u)\in \wt{\mc M}_{\mc C}(X,W; B)$, one has $u(\Sigma^*) \subset P \times_{\hat K} N$. 
\end{prop}

\begin{proof}
For any local coordinate $z = s + {\bm i} t$ on $\Sigma_{\mc C}^*$, let $\Delta$ be the standard Laplacian in this coordinate. Denote the volume form by $\sigma_K ds dt$. Denote $v_s = \partial_s u + {\mc X}_\phi$, $v_t = \partial_t u + {\mc X}_\psi$. By the vortex equation on $A_K$, one has
\beqn
\begin{split}
\nabla_s^A v_t - \nabla_t^A v_s = &\ \nabla_s (\partial_t u + {\mc X}_\psi) + \nabla_{\partial_t u + {\mc X}_\psi} {\mc X}_\phi - \nabla_t (\partial_s u + {\mc X}_\phi) - \nabla_{\partial_s u + {\mc X}_\phi} {\mc X}_\psi\\
                                = &\ \nabla_{\partial_s u} {\mc X}_\psi + {\mc X}_{\partial_s \psi} + \nabla_{\partial_t u + {\mc X}_\psi} {\mc X}_\phi - \nabla_{\partial_t u} {\mc X}_\phi - {\mc X}_{\partial_t \phi} - \nabla_{\partial_s u + {\mc X}_\phi} {\mc X}_\psi\\
																= &\ {\mc X}_{\partial_s \psi} - {\mc X}_{\partial_t \phi} + [{\mc X}_\psi, {\mc X}_\phi]\\
																= &\ {\mc X}_{F_A}\\
																= &\ {\mc X}_{F_{A_R}} - \sigma_K {\mc Y}_{\mu_K}.
\end{split}
\eeqn
Consider the proper function ${\mc F}_W = \mu_K \cdot \xi_W$ given by Hypothesis \ref{hyp24}, whose restriction to $X_W$ is $\hat{K}$-invariant. Moreover, $u$ is holomorphic with respect to $A$. Therefore
\beqn
\begin{split}
\Delta {\mc F}_W (u) &\ = \partial_s \langle \nabla {\mc F}_W (u), \partial_s u \rangle + \partial_t \langle \nabla {\mc F}_W (u), \partial_t u \rangle\\
                    &\ = \partial_s \langle \nabla {\mc F}_W (u), v_s \rangle + \partial_t \langle \nabla {\mc F}_W (u), v_t \rangle\\
                    &\ = \langle \nabla_s^A \nabla {\mc F}_W(u), v_s \rangle + \langle \nabla{\mc F}_W(u), \nabla_s^A  v_s \rangle + \langle \nabla_t^A \nabla {\mc F}_W(u), v_t \rangle + \langle \nabla {\mc F}_W(u), \nabla_t^A v_t \rangle \\
										&\ = \langle \nabla_{v_s} \nabla {\mc F}_W, v_s \rangle + \langle \nabla_{v_t} \nabla {\mc F}_W, v_t \rangle + \langle \nabla {\mc F}_W, - J \nabla^A_s v_t + J \nabla^A_t v_s \rangle \\
										&\ = \langle \nabla_{v_s} \nabla {\mc F}_W, v_s \rangle + \langle \nabla_{v_t} \nabla {\mc F}_W, v_t \rangle +  \langle \nabla {\mc F}_W, \sigma_K J {\mc Y}_{\mu_K} - J {\mc X}_{F_{A_R}} \rangle.
										\end{split}
\eeqn
Since $|F_{A_R}|$ is uniformly bounded, for some $\tau_R \in Z({\mf k})$, one has
\beqn
\Delta {\mc F}_W(u) \geq \langle \nabla_{v_s} \nabla {\mc F}_W, v_s \rangle + \langle \nabla_{v_t} \nabla {\mc F}_W, v_t \rangle + \langle \nabla {\mc F}_W, \sigma_K J {\mc Y}_{\mu_K - \tau_R} \rangle.
\eeqn
By Hypothesis \ref{hyp24}, whenever ${\mc F}_W (u(z)) \geq c_W( \tau_R)$, $\Delta {\mc F}_W(u) \geq 0$. Hence ${\mc F}_W (u)$ is a subharmonic function, which can only achieve its maximal value at punctures of $\Sigma_{\mc C}^*$. 

If ${\mc C}$ has at least one puncture, then by the asymptotic behavior of bounded solutions, at punctures the limit of $u$ is in the compact subset $X_W \cap \mu_K^{-1}(0)$, which cannot give arbitrarily large value of ${\mc F}_W$. Since ${\mc F}_W$ is proper, $u$ is contained in a compact subset. If ${\mc C}$ has no punctures, then ${\mc F}_W(u)$ is a constant function. However, since ${\mc F}_W = \mu_K \cdot \xi_W$, from the vortex equation, we know that 
\beqn
\langle * F_{A_K}, \xi_W\rangle = - \mu_K(u) \cdot \xi_W = - {\mc F}_W(u).
\eeqn
The integral of the left hand side over $\Sigma_{\mc C}^*$ is a topological quantity depending on $B$. Hence ${\mc F}_W(u)$ cannot be a very large constant.
\end{proof}

Thanks to the above $C^0$ bound, it is then routine to finish the proof of Theorem \ref{thm43}. We remark that because we used the cylindrical metric for the $K$-component of the vortex equation, the phenomenon that appeared in the approach of \cite{Mundet_Tian_2009} when conformal structure degenerates would not appear (see also \cite{Venugopalan_quasi}). The proof of Hausdorffness of $\ov{\mc M}_{g, k}(X, W; B)$ is also very similar to the case of Gromov--Witten theory. We omit the details.

\section{Abstract Theory of Virtual Cycles}\label{section5}

We recall the framework of constructing virtual fundamental cycles associated to moduli problems. Such constructions, usually called ``virtual technique,'' has a long history since it firstly appeared in algebraic Gromov--Witten theory by \cite{Li_Tian_2}. There have been many different approaches so far. The current method is based on the topological approach of \cite{Li_Tian}.

\subsection{Topological Virtual Orbifolds}

\subsubsection*{Orbifolds, orbibundles and embeddings}

We only consider topological orbifolds and continuous orbibundles. The notion of topological embeddings include the requirement that the embedding is locally flat. Further, the existence of normal bundle (or tubular neighborhood) is not automatic in the topological category.

Let $f: Y \to X$ be a topological embedding between orbifolds. A tubular neighborhood is a map $\pi: N \to Y$ where $N$ is a neighborhood of $f(Y)$ inside $X$ such that it contains another neighborhood $N'$ of $f(Y)$ such that $\pi|_{N'}: N' \to Y$ is a disk bundle over $Y$. Two tubular neighborhoods $\pi_1: N_1 \to Y$ and $\pi_2: N_2 \to Y$ are equivalent if there is another tubular neighborhood $\pi_3: N_3 \to Y$ with $N_3 \subset N_1 \cap N_2$ and $\pi_1|_{N_3} = \pi_2|_{N_3} = \pi_3$. An equivalence class of tubular neighborhoods is called a germ. A {\bf strong embedding} from $Y$ to $X$ is a topological embedding $f: Y \to X$ together with a germ of tubular neighborhoods $f^+$. An open embedding is automatically a strong embedding. Moreover, compositions of strong embeddings are still strong embeddings. 

Let $E_1 \to U_1$, $E_2 \to U_2$ be orbifold vector bundles. A bundle embedding of $E_1 \to U_1$ into $E_2 \to U_2$ is a pair $(\phi, \wh\phi)$ where $\phi: U_1 \to U_2$ is a topological embedding, and $\wh\phi: E_1 \to E_2$ is an injective bundle map which covers $\phi$. A tubular neighborhood of the bundle embedding $(\phi, \wh\phi)$ is a tubular neighborhood $\nu: N \to \phi(U_1)$ of $\phi: U_1 \to U_2$ as well as a projection
\beqn
\wh\nu: E_2|_N \to E_2|_{\phi(U_1)}
\eeqn
whose restriction to $\phi(U_1)$ is the identity. Two tubular neighborhoods $(\nu: N \to \phi(U_1), \wh\nu)$ and $(\nu': N' \to \phi(U_1), \wh\nu')$ are equivalent if they have a common, smaller tubular neighborhoods. An equivalence class of tubular neighborhoods is also called a germ. A strong embedding of orbifold vector bundles is a bundle embedding together with a germ of tubular neighborhoods. Strong bundle embeddings can also be composed.

\subsubsection*{Multisections and perturbations}

When involved with group actions, it is generally impossible to achieve transversality while preserving the equivariance. A typical way to handle the equivariant situation is to use multi-valued perturbations.

Let $A$, $B$ be sets, $l\in {\mb N}$, and ${\mc S}^l(B)$ be the $l$-fold symmetric product of $B$. An {\bf $l$-multimap} $f$ from $A$ to $B$ is a map $f: A \to {\mc S}^l(B)$. For another $l' \in {\mb N}$, there is a natural map $m_{l'}: {\mc S}^l(B) \to {\mc S}^{ll'} (B)$. Hence $f$ can be identified with an $ll'$-multimap $m_{l'} \circ f$. If $\Gamma= \{ g_1, \ldots, g_s\}$ acts on $A$ and $B$, then it acts on ${\mc S}^l(B)$ for any $l\in {\mb N}$ and we can talk about $\Gamma$-equivariant multimaps from $A$ to $B$.

One can define the notion of multisections of an orbibundle $E \to X$ using local charts and transition functions. There is also a notion of topological transversality for multisections. Finally, since one can achieve transversality in the topological category by small perturbations (see \cite{Quinn}), one can perturb (single-valued) sections to transverse multisections.

\subsubsection*{Virtual orbifold atlases}

In this subsection, all orbifolds can be realized as effective global quotients as $M/\Gamma$ where $M$ is a topological manifold and $\Gamma$ is a finite group. 

Let $X$ be a compact and Hausdorff space. A {\bf virtual orbifold chart} (chart for short) on $X$ is a tuple $K = (U, E, S, \psi, F)$ where $U$ is an orbifold, $E \to U$ is an orbifold vector bundle, $S: U \to E$ is a continuous section, $\psi: S^{-1}(0) \to X$ is a homeomorphism onto its image $F\subset X$. $F$ is called the {\bf footprint} of the chart $K$. The integer ${\rm dim}^{vir} K:= {\rm dim} U - {\rm rank} E$ is called the virtual dimension of $K$. If $U' \subset U$ is a precompact open subset, then we can restrict $K$ to $U'$ in the obvious way, denoted by $K |_{U'}$, called a {\bf shrinking} of $K$. 

\begin{defn}\label{defn51}
Let $K_i:= (U_i, E_i, S_i, \psi_i, F_i)$, $i=1,2$ be two charts of $X$. An {\bf embedding} of $K_1$ into $K_2$ consists of a strong embedding ${\bm \phi}_{21}:= (\phi_{21}, \wh{\phi}_{21}, \phi_{21}^+, \wh\phi_{21}^+)$ of orbifold vector bundles and a projection $\wh\pi_{21}: E_2|_{\phi_{21}(U_1)} \to \wh\phi_{21}(E_1)$, such that the diagrams
\beqn
\xymatrix{E_1 \ar[r]^{\wh\phi_{21}}   \ar[d]^{\phi_1} & E_2 \ar[d]_{\phi_2} \\
          U_1 \ar@/^1pc/[u]^{S_1} \ar[r]^{\phi_{21}} & U_2 \ar@/_1pc/[u]_{S_2}       }\ \ \ \ \ \xymatrix{  S_1^{-1}(0) \ar[r]^{\phi_{21}} \ar[d]^{\psi_1} & S_2^{-1}(0) \ar[d]^{\psi_2}  \\     X \ar[r]^{{\rm Id}} & X }
\eeqn
commute. Moreover, denote $E_{21}:= {\rm ker} (\wh\pi_{21})\subset E_2|_{\phi_{21}(U_{1})}$ which is complementary to $\wh\phi_{21}(E_1)$. For any tubular neighborhood $(\nu_{21}: N_{21} \to \phi_{21}(U_1), \wh\nu_{21})$ representing the germ $(\phi_{21}^+, \wh\phi_{21}^+)$, consider the composition
\beqn
\xymatrix{
N_{21} \ar[r]^-{S_2} & E_2|_{N_{21}} \ar[r]^-{\wh\nu_{21}} &  E_2|_{\phi_{21}(U_1)} \ar[r]^-{{\rm Id} - \wh\pi_{21}} & E_{21}}
\eeqn
which preserves fibres. We require that its restriction to each fibre is a homeomorphism near the origin.
\end{defn}

One can prove that the composition of two embeddings is still an embedding. 

\begin{defn}\label{defn52}
Let $K_i = (U_i, E_i, S_i, \psi_i, F_i)$, $(i=1, 2)$ be two charts. A {\bf weak} coordinate change from $K_1$ to $K_2$ is a triple $T_{21} = (U_{21}, {\bm \phi}_{21}, \wh\pi_{21}) $, where
\begin{enumerate}
\item $U_{21}\subset U_1$ is an open subset.

\item $({\bm \phi}_{21}, \wh\pi_{21})$ is an embedding from $K_1|_{U_{21}}$ to $K_2$.

\indent\noindent It is called a {\bf strong} coordinate change if in addition

\item $\psi_1( (S_1)^{-1}(0) \cap U_{21}) = F_1\cap F_2  \subset X$.
\end{enumerate}
\end{defn}

\begin{lemma}\label{lemma53}
Let $K_i = (U_i, E_i, S_i, \psi_i, F_i)$, $(i=1, 2)$ be two charts and let $T_{21} = (U_{21}, {\bm \phi}_{21})$ be a weak coordinate change from $K_1$ to $K_2$. Suppose $K_i' = K_i|_{U_i'}$ be a shrinking of $K_i$. Then the restriction $T_{21}':= T_{21}|_{U_1' \cap \phi_{21}^{-1}(U_2')}$ is a weak coordinate change from $K_1'$ to $K_2'$. Moreover, if $T_{21}$ is strong, then $T_{21}'$ is also strong.
\end{lemma}

\begin{defn}\label{defn54}
Let $X$ be a compact Hausdorff space. A weak (resp. strong) {\bf virtual orbifold atlas} of virtual dimension $k$ on $X$ is a collection
\beqn
{\mf A}:= \Big( \Big\{ K_I:= (P_I, \psi_I, F_I)\ |\ I \in {\mc I} \Big\},\ \Big\{ T_{JI} = \big( P_{JI}, {\bm \phi}_{JI} \big) \ |\ I \leq J \Big\}\Big),
\eeqn
where $({\mc I}, \leq)$ is a finite, partially ordered set; for each $I\in {\mc I}$, $K_I$ is a chart of virtual dimension $k$; and for each pair $I \leq J$, $T_{JI} = (U_{JI}, {\bm \phi}_{JI} = (\phi_{JI}, \wh\phi_{JI}, \phi_{JI}^+))$ is a weak (resp. strong) coordinate change from $K_I$ to $K_J$. They are subject to the following conditions.
\begin{itemize}
\item {\bf (Cocycle Condition)} For $I \leq J \leq K \in {\mc I}$, denote $U_{KJI} = U_{KI} \cap \phi_{JI}^{-1} (U_{KJ}) \subset U_I$, then we require
\beqn
{\bm \phi}_{KI}|_{U_{KJI}} = {\bm \phi}_{KJ} \circ {\bm \phi}_{JI}|_{U_{KJI}}.
\eeqn
More explicitly, denote $E_{I, KJI} = E_I|_{U_{KJI}}$. Then the cocycle condition means
\beqn
\begin{split}
\phi_{KI}|_{U_{KJI}} = &\ \phi_{KJ}\circ \phi_{JI}|_{U_{KJI}},\\
\wh\phi_{KI}|_{E_{I, KJI}} = &\ \wh\phi_{KJ} \circ \wh\phi_{JI}|_{E_{I, KJI}}.
\end{split}
\eeqn

\item {\bf (Filtration Condition)} For $I\leq J\leq K \in {\mc I}$, we have,
\beqn
\wh\pi_{KI}|_{\phi_{KI}(U_{KJI})} =  \wh\pi_{JI}|_{\phi_{JI}(U_{KJI})}  \circ \wh\pi_{KJ}|_{\phi_{KI}(U_{KJI})}
\eeqn
Equivalently, if we denote $E_{JI} = {\rm ker} \wh\pi_{JI}$, then
\beq\label{eqn51}
E_{KI}|_{\phi_{KI}(U_{KJI})} = E_{KJ}|_{\phi_{KJ}(U_{KJI})} \oplus \wh\phi_{KJ} ( E_{JI}|_{\phi_{JI}(U_{KJI})}).
\eeq

\item {\bf (Overlapping Condition)} For $I, J \in {\mc I}$, $\ov{F_I} \cap \ov{F_J} \neq \emptyset \Longrightarrow I \leq J\ {\rm or}\ J \leq I$.
\end{itemize}
\end{defn}

All virtual orbifold atlases considered in this paper have definite virtual dimensions, although sometimes we do not explicitly mention it.

\begin{notation}
If $\wh\nu_{JI}: E_J|_{N_{JI}} \to E_J|_{\phi_{JI}(U_{JI})}$ is a representative of $\wh\phi_{JI}^+$, then it induces an isomorphism
\beqn
\wh\varphi_{JI}: E_J|_{N_{JI}} \simeq \nu_{JI}^* \left( E_J|_{\phi_{JI}(U_{JI})} \right).
\eeqn
On the other hand, the projection $\wh\pi_{JI}: E_J|_{\phi_{JI}(U_{JI})} \to \wh\phi_{JI}(E_I|_{U_{JI}})$ induces a splitting
\beqn
E_J|_{\phi_{JI}(U_{JI})} \simeq \wh\phi_{JI}(E_I|_{U_{JI}}) \oplus E_{JI}.
\eeqn
It can be extended to the tubular neighborhood $N_{JI}$, denoted by
\beq\label{eqn52}
\wh\varphi_{JI}: E_J|_{N_{JI}} \simeq \nu_{JI}^* \wh\phi_{JI} (E_I |_{U_{JI}}) \oplus \nu_{JI}^* E_{JI}	
\eeq
\end{notation}

\subsection{Virtual fundamental cycle}

In order to construct the virtual cycle, we need some more technical preparations. 

\begin{defn}
Let ${\mf A}:= ( \{ K_I = (U_I, E_I, S_I, \psi_I, F_I)\ |\ I \in {\mc I} \},\ \{ T_{JI}= (U_{JI}, {\bm \phi}_{JI})\ |\ I \leq J \})$ be a strong virtual orbifold atlas on $X$. A {\bf shrinking} of ${\mf A}$ is another virtual orbifold atlas ${\mf A}' = ( \{ K_I' = (U_I', E_I', S_I', \psi_I', F_I')\ |\ I \in {\mc I} \},\ \{ T_{JI}' = (U_{JI}', {\bm \phi}_{JI}')\ |\ I \leq J \})$ indexed by elements of the same set ${\mc I}$ such that 
\begin{enumerate}
\item For each $I \in {\mc I}$, $K_I'$ is a shrinking $K_I|_{U_I'}$ of $K_I$.

\item For each pair $I \leq J$, $T_{JI}'$ is the induced shrinking of $T_{JI}$ given by Lemma \ref{lemma53}.
\end{enumerate}
\end{defn}

Given a strong virtual orbifold atlas ${\mf A}:= ( \{ K_I = (U_I, E_I, S_I, \psi_I, F_I)\ |\ I \in {\mc I} \},\ \{ T_{JI}= (U_{JI}, {\bm \phi}_{JI})\ |\ I \leq J \} )$, 
we define a relation $\curlyvee$ on the disjoint union $\bigsqcup_{I \in {\mc I}} U_I$ as follows. $U_I \ni x \curlyvee y\in U_J$ if one of the following holds.
\begin{enumerate}
\item $I = J$ and $x = y$;

\item $I \leq J$, $x \in U_{JI}$ and $y = \phi_{JI}(x)$;

\item $J \leq I$, $y \in U_{IJ}$ and $x = \phi_{IJ}(y)$.
\end{enumerate}
If ${\mf A}'$ is a shrinking of ${\mf A}$, then it is easy to see that the relation $\curlyvee'$ on $\bigsqcup_{I \in {\mc I}} U_I'$ defined as above is induced from the relation $\curlyvee$ for ${\mf A}$ via restriction. Then one can prove the following result (see the proof in \cite[Appendix]{Tian_Xu_3}).

\begin{lemma}\label{lemma57}
For each virtual orbifold atlas ${\mf A}$, there exists a shrinking ${\mf A}'$ of ${\mf A}$ such that $\curlyvee'$ is an equivalence relation. Moreover, if we define $|{\mf A}'|:= \Big( \bigsqcup_{I \in {\mc I}} U_I' \Big)/ \curlyvee$, then one can make $|{\mf A}'|$ Hausdorff and for each $I \in {\mc I}$, the natural map $U_I' \to |{\mf A}'|$ is a homeomorphism onto its image.
\end{lemma}

The above conditions for the shrunk atlas ${\mf A}'$ also appear in the definition of the so-called ``good coordinate system'' in the theory of Kuranishi structures (see \cite{FOOO_14}).

\begin{thm}\label{thm58}
Let $X$ have a virtual orbifold atlas
\beqn
{\mf A} = \Big( \big\{ K_I = (P_I, \psi_I, F_I) \ |\ I \in {\mc I} \big\},\ \big\{ T_{JI} = (P_{JI}, {\bm \phi}_{JI}, \wh\pi_{JI}) \ |\ I \leq J \big\} \Big)
\eeqn
which satisfies the conditions of Lemma \ref{lemma57}. Then for any $\epsilon>0$, there exist multisections $t_I^\epsilon$ of $E_I \to U_I$ for all $I \in {\mc I}$ satisfying the following conditions.
\begin{enumerate}
\item For $I \leq J$, there is a tubular neighborhood $\wh\nu_{JI}: E_J|_{N_{JI}} \to E_J|_{\phi_{JI} (U_{JI})}$ representing the germ $\wh\phi_{JI}^+$ such that for $z_J \in N_{JI}$,
\beqn
\wh\varphi_{JI} (z_J) \Big( t_J^\epsilon(z_J) \Big) = \Big(  \wh\phi_{JI}  \big( t_I^\epsilon( \phi_{JI}^{-1} ( \nu_{JI} (z_J))) \big), 0_{JI} \Big).
\eeqn
Here we used the notation \eqref{eqn52} and $0_{JI}\in \nu_{JI}^* E_{JI}$ is the zero section.

\item For every $I \in {\mc I}$, $\| t_I^\epsilon \|_{C^0} \leq \epsilon$.

\item For every $I \in {\mc I}$, $S_I^\epsilon:= S_I + t_I^\epsilon$ is a transverse multisection.
\end{enumerate}
Denote by ${\mf S}^\epsilon$ the system $\{ S_I^\epsilon\}_{I \in {\mc I}}$ of perturbations and $({\mf S}^\epsilon)^{-1}(0):= \left(  \bigsqcup_{I \in {\mc I}}(S_I^\epsilon)^{-1}(0)\right)/\curlyvee \subset |{\mf A}|$. Then for $\epsilon$ sufficiently small, $({\mf S}^\epsilon)^{-1}(0)$ is compact.
\end{thm}

\subsubsection*{Strongly continuous maps} Consider a compact Hausdorff space $X$ equipped with a strong virtual orbifold atlas ${\mf A}$ that satisfies the conditions of Lemma \ref{lemma57}. Let $Y$ be a topological space. A {\bf strongly continuous map} from $(X, {\mf A})$ to $Y$ consists of a collection of continuous maps $f_I: U_I \to Y$ such that for all $I \leq J$, $f_J \circ \phi_{JI} = f_I$. Assume that ${\mf S}$ is a transverse multi-valued perturbation of ${\mf A}$ and ${\mf S}^{-1}(0)$ is compact in $|{\mf A}|$. We skip the discussion of orientations but our problem induces a natural orientation, so that we can define the push-forward class 
\beqn
f_*[ {\mf S}^{-1}(0)] \in H_*(Y; {\mb Q})
\eeqn
One can also prove that two different systems of choices (such as the choice of perturbations) give the same class. In application, the strongly continuous map is usually the evaluation map, which is defined not only on the moduli space $X$, but also on the thickening $|{\mf A}|$.

\section{Definition of the Invariants}\label{section6}

Now we describe how to construct a virtual orbifold atlas on $\ov{\mc M}_{g, k}(X, W; B)$ and to define the CohFT. One first needs to study the linear theory associated to the gauged Witten equation, including a computation of Fredholm index.

\subsection{The Fredholm theory and expected dimensions}

Let $k \geq 0$ be the number of punctures and $g \geq 0$ be the genus of the domain curve. Let $[\hat\gamma_1], \ldots, [\hat\gamma_k]$ be conjugacy classes of $K$ that have finite orders and choose representatives $\hat\gamma_1, \ldots, \hat\gamma_k$. By abuse of notation, $[\hat\gamma_j]$ also carries the information of an integer $m_j \in \{0, 1, \ldots, r-1\}$ that prescribes the type of the puncture $z_j$ of the $r$-spin structure. 

\begin{defn}
Let $\wt{\mc M}_{g, k}([\hat\gamma_1], \ldots, [\hat\gamma_k])$ be the set of triples $({\mc C}, {\bf v})$, where ${\mc C}$ represents a point of ${\mc M}{}_{g, k}^r(m_1, \ldots, m_k)$ and ${\bf v} = (A, u)$ is a smooth bounded solution to the gauged Witten equation \eqref{eqn33} over $\Sigma_{\mc C}^*$ such that the limiting monodromy of $A$ at the $j$-th marked point is $[\hat\gamma_j]$. Let ${\mc M}_{g, k}([\hat\gamma_1], \ldots, [\hat\gamma_k])$ be the quotient of $\wt{\mc M}_{g, k} ([\hat\gamma_1], \ldots, [\hat\gamma_k])$ by the group of smooth gauge transformations over the punctured surfaces.
\end{defn}

Each gauge equivalence class $[{\mc C},{\bf v}] \in {\mc M}_{g, k}([\hat\gamma_1], \ldots, [\hat\gamma_k])$ represents a class class $B \in H_2^K(X_W; {\mb Q})$. Then one can decompose the moduli space into connected components:
\beqn
{\mc M}_{g, k}([\hat\gamma_1], \ldots, [\hat\gamma_k]) = \bigsqcup_{B \in H_2^K(X_W; {\mb Q})} {\mc M}_{g, k}([\hat\gamma_1], \ldots, [\hat\gamma_k]; B). 
\eeqn
Meanwhile, we have a well-defined continuous evaluation map
\beqn
\ev: {\mc M}_{g, k}([\hat\gamma_1], \ldots, [\hat\gamma_k]) \to \prod_{j=1}^k \bar{X}_{W, [\hat\gamma_j]}.
\eeqn
	
We would like to realize ${\mc M}_{g, k}([\hat\gamma_1], \ldots, [\hat\gamma_k])$ as the zero set of certain Fredholm section. More precisely, fix $p>2$ and a small $\delta > 0$. Then define
\beq\label{eqn61}
{\mc B}_{\mc C} = {\mc B}_{{\mc C}, R}^{1, p, \delta} \oplus W^{1, p, \delta} \Big( \Sigma_{{\mc C}}^*, {\rm ad}(P_K)|_{\Sigma_{{\mc C}}^*} \Big) \oplus W^{1, p, \delta} \Big( \Sigma_{{\mc C}}^*, u^* Y_{{\mc C}}^* \Big) \oplus T_{\mc C},
\eeq
here ${\mc B}_{{\mc C}, R}^{1, p, \delta} \subset W^{1, p, \delta}\big( \Sigma_{{\mc C}}^*, {\rm ad}(P_R)|_{\Sigma_{{\mc C}}^*}\big)$ is the space of infinitesimal deformations of unitary connections on $L_R$ that do not change the holomorphic structure, and $T_{\mc C}$ is a finite dimensional vector space consisting of variations of the evaluations of ${\bf v}$ at the punctures.

Let ${\mc G}_{\mc C}:= W^{2, p, \delta} \big( \Sigma_{{\mc C}}^*, {\rm ad} (P_{\mc C}^* )\big)$ be the space of infinitesimal gauge transformations on ${\mc C}$. Denote
\beqn
{\mc E}_{\mc C}:= L^{p, \delta} \Big( \Sigma_{{\mc C}}^*, {\rm ad}(P_{\mc C}^*) \Big) \oplus L^{p, \delta} \Big( \Sigma_{{\mc C}}^*, u^* T^{\rm vert} Y_{{\mc C}}^* \Big).
\eeqn
Then the linearization of the gauged Witten equation over ${\mc C}$ and the linearization of the gauge transformation give a complex of Banach spaces
\beq\label{eqn62}
\xymatrix{ {\mc G}_{\mc C} \ar[r] & {\mc B}_{\mc C} \ar[r] & {\mc E}_{\mc C}}.
\eeq

\begin{prop}\label{prop62}
\eqref{eqn62} is a Fredholm complex and its Euler characteristic is  
\beqn
{\rm dim}_{\mb C} \bar{X}(2-2g) + 2 c_{\, 1}^K (B)  -  \sum_{j=1}^k \Big[ 2 \iota ([\hat\gamma_j]) + {\rm rank}_{\mb C} \nabla^2 W|_{\bar{X}_{W, [\hat\gamma_j]}} \Big].
\eeqn
Therefore the expected dimension of ${\mc M}_{g, k}([\hat\gamma_1], \ldots, [\hat\gamma_k]; B)$ is equal to 
\beqn
\big[ {\rm dim}_{\mb C} \bar{X} - 3 \big] (2-2g) + 2k  + 2 c_{\, 1}^K (B)  -  \sum_{j=1}^k \Big[ 2 \iota ([\hat\gamma_j]) + {\rm rank}_{\mb C} \nabla^2 W|_{\bar{X}_{W, [\hat\gamma_j]}} \Big].
\eeqn
\end{prop}

When proving Proposition \ref{prop62} it is convenient to introduce a Banach space ${\mc E}_{\mc C}^+:= {\mc E}_{\mc C} \oplus {\mc E}_{\mc C}'$, where ${\mc E}_{\mc C}'= L^{p, \delta}\big( \Sigma_{\mc C}^*, {\rm ad}(P_{\mc C}^*) \big)$. Consider an ``augmented gauged Witten equation'' ${\mc F}({\bf v}') = 0 \in {\mc E}_{\mc C}^+$, which is defined for objects ${\bf v}'$ near ${\bf v}$ and the ${\mc E}_{\mc C}'$ component of ${\mc F}({\bf v}')$ is a local gauge-fixing condition. Its linearized operator 
\beq\label{eqn63}
{\mc D}_{\mc C}: {\mc B}_{\mc C} \to {\mc E}_{\mc C}^+
\eeq
is Fredholm and its index is equal to the Euler characteristic of the complex \eqref{eqn62}.

Meanwhile, the moduli space $\ov{\mc M}{}_{g, k}^r([\hat\gamma_1], \ldots, [\hat\gamma_k];B)$ has lower strata parametrized by dual graphs. A decorated dual graph $\sf \Gamma$ is a connected graph, consisting of a set of vertices $\sf V(\Gamma)$, a set of edges $\sf E (\Gamma)$ and a set of tails $\sf T(\Gamma)$, such that each ${\sf v} \in \sf V(\Gamma)$ is decorated with a class $B_{\sf v} \in H_2^K(X_W; {\mb Q})$ and a genus $g_{\sf v} \geq 0$. Given a decorated dual graph $\sf \Gamma$, ${\mc M}_{\sf \Gamma}([\gamma_1], \ldots, [\gamma_k])$ is the fibred product defined by the matching condition via the evaluation maps at nodes. 

\begin{prop}
Given a decorated dual graph $\sf \Gamma$, the expected dimension of the moduli space ${\mc M}_{\sf \Gamma} ([\gamma_1], \ldots, [\gamma_k])$ is equal to 
\beqn
\big[ {\rm dim}_{\mb C} \bar{X} - 3 \big] (2-2g) + 2k + 2c_{\, 1}^K(B) - \sum_{j=1}^k \Big[ 2 \iota ([\gamma_j]) + {\rm rank}_{\mb C}\nabla^2 W|_{\bar{X}_{W, j}} \Big] - 2 \# \sf E(\Gamma).
\eeqn
\end{prop}

This result follows formally from Proposition \ref{prop62} and the following property of $\iota([\hat\gamma])$. 
\begin{prop}
For any conjugacy class $[\hat\gamma]$ of $K$ of finite order such that $\bar{X}_{W, [\hat\gamma]} \neq \emptyset$,
\beqn
\iota([\hat\gamma]) + \iota([\hat\gamma^{-1}]) = {\rm codim}_{\mb C} \Big[ \bar{X}_{W, [\hat\gamma]}, \bar{X}_W \Big].
\eeqn
\end{prop}

\subsection{Constructing virtual cycles}

We give a short account on how to construct a virtual cycle satisfying the axioms. The detail will appear in \cite{Limit1}.	

\subsubsection*{Constructing local charts}

The first job is to construct a local chart around each point of the compactified moduli space $\ov{\mc M}_{g, k}(X, W; B)$. To illustrate the construction, assume for simplicity that we have a stable solution over a {\it stable} curve with irreducible components ${\mc C}_\alpha$; in particular, there is no soliton components. The solution represents a point $p \in \ov{\mc M}_{g, k}(X, W; B)$. Moreover, assume that the solution has trivial automorphism. 

Let ${\mc B}_\alpha$ be the Banach space ${\mc B}_{\mc C}$ in \eqref{eqn61} for with ${\mc C}$ replaced by ${\mc C}_\alpha$. Define
\beqn
{\mc B}:= \Big\{ (\xi_\alpha) \in \bigoplus_\alpha {\mc B}_\alpha\ |\ {\rm matching\ condition} \Big\}.
\eeqn
Here ``matching condition'' means at each node $w$ connecting two components ${\mc C}_\alpha$ and ${\mc C}_{\alpha'}$ the component of $\xi_\alpha$ in $T_{{\mc C}_\alpha}$ and the component of $\xi_{\alpha'}$ in $T_{{\mc C}_{\alpha'}}$ are equal. Denote
\beqn
{\mc E}:= \bigoplus_\alpha {\mc E}_\alpha^+.
\eeqn
Therefore, the restriction of the direct sum of all ${\mc D}_{{\mc C}_\alpha}$ (see \eqref{eqn63}) to ${\mc B}$, denoted by ${\mc D}: {\mc B} \to {\mc E}$ is a Fredholm operator. 

From the ellipticity of ${\mc D}$, one can fine a finite dimensional subspace $E:= \bigoplus_\alpha E_\alpha \subset {\mc E}$, generated by smooth sections that have compact supports away from the nodes and punctures. (When there is a nontrivial automorphism group, one has to require that $E$ is invariant under the automorphism group.) Then define
\beq\label{eqn64}
U_{\rm map}:= \Big\{ (e_\alpha, \xi_\alpha)\in E \times {\mc B}\ |\ e_\alpha + {\mc D}_\alpha(\xi_\alpha) = 0 \Big\}.
\eeq
We still use $U_{\rm map}$ to denote a small neighborhood of the origin of the above vector space.

Moreover, introduce the space of smoothing parameters $U_{\rm def}$ which is a small ball centered at the origin of the universal infinitesimal deformation space of ${\mc C}$. Then $\zeta \in U_{\rm def}$ parametrizes a family of stable $r$-spin curves ${\mc C}_\zeta$ that represent points in a neighborhood of $[{\mc C}]$ in $\ov{\mc M}{}_{g, k}^r$. Moreover, there is a subspace $U_{\rm def}' \subset U_{\rm def}$ parametrizing nearby curves that have the same combinatorial type as ${\mc C}$, i.e., without smoothing out any nodes. Let $U_{\rm def}''$ be the complement which is homeomorphic to a product of $k$ small disks in ${\mb C}$ where $k$ is equal to the number of nodes of ${\mc C}$. Let a general element of $U_{\rm def}$ be $\zeta = (\zeta', \zeta'')$. 

Then by a standard application of the implicit function theorem for Banach spaces, one obtain a family of objects $(\wt{e}_\alpha, \wt{\bf v}_\alpha, {\mc C}_{\zeta'})$ over the nodal curve ${\mc C}_{\zeta'}$, parametrized by elements in $U_{\rm map} \times U_{\rm def}'$ where $(\wt{e}_\alpha, \wt{\bf v}_\alpha)$ is a solution to $\wt{e}_\alpha + {\mc F}_{{\mc C}_\alpha}(\wt{\bf v}_\alpha) = 0$.

Now turning on the smoothing parameter. For $\zeta = (\zeta', \zeta'')$ with $\zeta'' \neq 0$, one can construct a family of appropriate solutions $(\wt{e}^{\rm app}, \wt{\bf v}^{\rm app})$ over ${\mc C}_\zeta$ of the equation $\wt{e}^{\rm app} + {\mc F}_{{\mc C}_\zeta}(\wt{\bf v}^{\rm app}) = 0$. They are not actual solutions, but by proving a few estimates and applying the implicit function theorem, one can show that each $(\wt{e}^{\rm app}, \wt{\bf v}^{\rm app})$ can be corrected (in a specific way) to an actual solution. For example, one needs to prove 
\begin{prop}
There is a constant $\tau>0$ such that for $\zeta$ sufficiently small, 
\beqn
\Big\|\wt{e}^{\rm app} + {\mc F}_{{\mc C}_\zeta}(\wt{\bf v}^{\rm app}) \Big\|_{L^{p, \delta}(\Sigma_{{\mc C}_\zeta}^*)} \leq |\zeta''|^\tau.
\eeqn
\end{prop}
Moreover, the ``specific way'' of correcting an appropriate solution to an actual solution depends on the choice of a right inverse of the linearization of the augmented gauged Witten equation. Let the actual solution be denoted by $(\wt{e}^{\rm exact}, \wt{u}^{\rm exact})$, which is parametrized by $U:= U_{\rm map} \times U_{\rm def}$. There is an obvious map $S: U \to E$, and an obvious map $\psi: S^{-1}(0) \to \ov{\mc M}_{g,k}(X, W)$. One also need to prove the following results
\begin{lemma}
For sufficiently small shrinking of $U_{\rm map} \times U_{\rm def}$, the map $\psi$ is a homeomorphism onto its image and covers an open neighborhood $F$ of $[{\mc C}, {\bf v}]$ in $\ov{\mc M}_{g, k}(X, W)$. 
\end{lemma}

Therefore, the tuple $K = (U, E, S, \psi, F)$ defines a virtual orbifold chart in the sense of Definition \ref{defn51} (with trivial isotropy group).

\subsubsection*{Gluing charts}

The local charts $K_p$ we constructed have no transition functions among them. Now we explain the basic idea of how to construct an atlas that have transition functions out of these local charts. Assume for simplicity that $\ov{\mc M}_{g, k}(X, W; B)$ has only two strata, the top stratum ${\mc M}_0$ and a lower stratum ${\mc M}_1$, where the latter is compact. Then ${\mc M}_1$ can be covered by the footprints $F_{p_i}$ of a finite collection of the charts $K_{p_1}, \ldots, K_{p_n}$. Assume furthermore that $n=2$; one can use either a slighly modified argument or induction to treat the general case. Shrink $F_{p_1}$, $F_{p_2}$ to open subsets $F_{p_1}' \sqsubset F_{p_1}$ and $F_{p_2}' \sqsubset F_{p_2}$ which still cover ${\mc M}_1$. Here $A \sqsubset B$ means $A \subset B$ and the closure of $A$ is compact. Denote $E_+ = E_{p_1} \oplus E_{p_2}$. Then one can perform the gluing construction for the bigger obstruction space $E_+$, and construct a chart $K_+ = (U_+, E_+, S_+, \psi_+, F_+)$, where $F_+$ covers a neighborhood of $\ov{F_{p_1}'}  \cap \ov{F_{p_2}'} \cap {\mc M}_1$. Furthermore, since $E_{p_i} \subset E_+$ for $i = 1, 2$, one can construct weak coordinate changes from $K_{p_i}$ to $K_+$. This is because in \eqref{eqn64}, enlarging the obstruction space $E$ results a bigger solution set $U_{\rm map}$. Then the collection $(K_{p_1}, K_{p_2}, K_+)$ together with the coordinate changes satisfies the cocycle condition and the filtration condition in Definition \ref{defn54}. 

One has to do one more step of modifications so that the atlas satisfies the overlapping condition in Definition \ref{defn54}. Indeed, take shrinkings $F_{p_i}'''  \sqsubset F_{p_i}'' \sqsubset F_{p_i}'$ such that $F_{p_1}'''$ and $F_{p_2}'''$ still cover ${\mc M}_1$. Then we can shrink $K_{p_1}$ to $K_{p_1}^*$ whose footprint is $F_{p_1}''' \setminus \ov{F_{p_2}''}$, shrink $K_{p_2}$ to $K_{p_2}^*$ whose footprint is $F_{p_2}''' \setminus \ov{F_{p_1}''}$, and shrink $K_+$ to $K_+^*$ whose footprint is $F_{p_1}' \cap F_{p_2}'$. Then it is easy to see that $K_{p_1}^*$, $K_{p_2}^*$ and $K_+^*$ satisfy the overlapping condition. 

Now the three charts cover an open neighborhood ${\mc U}_1$ of ${\mc M}_1$, whose complement is compact. Then one can find another finite collection of charts $K_{q_i}$ whose footprints cover ${\mc M}_0 \setminus {\mc U}_1$. It is similar to carry out the construction inductively, together with shrinking the charts for a few times. We skip the details. We can also make sure (by shrinking charts appropriately) that the coordinate changes are strong in the sense of Definition \ref{defn52}. Finally one constructed a strong virtual orbifold atlas on $\ov{\mc M}_{g, k}(X, W; B)$.

\subsection{The invariants}

Now we have constructed a strong atlas ${\mf A}$ on the moduli space $\ov{\mc M}_{g, k}(X, W; B)$. By results recalled in Section \ref{section5}, by shrinking, we may assume that ${\mf A}$ satisfies the conditions of Lemma \ref{lemma57}. Restricting to components labelled by monodromies at punctures, the evaluation map and the forgetful map
\beqn
\ov{\mc M}_{g, k}([\hat\gamma_1], \ldots, [\hat\gamma_k]; B) \to \ov{\mc M}_{g, k} \times \prod_{j=1}^k \bar{X}_{W, [\hat\gamma_j]}
\eeqn
can be extended to a strongly continuous map defined on ${\mf A}$. Then for any transverse multi-valued perturbation ${\mf S}$ of ${\mf A}$, the pushforward of ${\mf S}^{-1}(0)$ defines a rational homology class 
\beqn
[\ov{\mc M}_{g, k}([\hat\gamma_1], \ldots, [\hat\gamma_k]; B)]^{\rm vir} \in H_*( \ov{\mc M}_{g, k}; {\mb Q}) \otimes \prod_{j=1}^k H_* ( \bar{X}_{W, [\hat\gamma_j]}; {\mb Q}).
\eeqn
This class is independent of the choice of perturbations and various other choices. Evaluating against cohomology classes, and summing over all components of $\ov{\mc M}_{g, k}(X, W; B)$, this gives the definition of genus $g$ correlation functions
\beqn
\langle \sqbullet \rangle_{g, k}: {\mc H}^{\otimes k} \times H^*(\ov{\mc M}_{g, k}; {\mb Q}) \to \Lambda.
\eeqn

Lastly, the system of correlation functions we just defined satisfy the two splitting axioms (see Definition \ref{defn25}). This can be shown by essentially the same methods as proving the splitting axioms of Gromov--Witten invariants, although we have to make the argument adapted to the current setting. (See \cite{Ruan_Tian_97} for the case of Gromov--Witten theory in the semi-positive case and \cite{Fukaya_Ono} for the general case.)

\section{Discussions}\label{section7}

We discuss several related problems, some of which are more or less at hand and some of which are still difficult to deal with.

\subsection{The Landau--Ginzburg theories}

As the limit of GLSM, the (orbifold) Landau--Ginzburg A-model theory has been constructed by Fan--Jarvis--Ruan \cite{FJR3, FJR2}. The construction of virtual cycles on the moduli space of the Witten equation is based on the theory of Kuranishi structure of Fukaya--Ono \cite{Fukaya_Ono}. The narrow case also has an algebraic construction, due to Chang--Li--Li \cite{Chang_Li_Li} using the cosection construction of virtual cycles developed in \cite{Kiem_Li}. 

In GLSM for $W = pQ$ where $Q: {\mb C}^N \to {\mb C}$ is a quasihomogeneous polynomial, one can turn to the non-geometric phase which can recover the Landau--Ginzburg theory. Given an $r$-spin curve $({\mc C}, L_R)$ and another line bundle $L_K$, the gauged Witten equation is roughly of the following form
\beq\label{eqn71}
\left\{ \begin{array}{rc}
\ov\partial_A u + \ov{p} \nabla Q(u) = &\ 0,\\
\ov\partial_A p + \ov{Q(u)} = &\ 0,\\
F_{A_K} - {\bm i} \Big[ \sum_{i=1}^N r_i|u_i|^2 - r|p|^2 - c_K \Big] \sigma_c = &\ 0.
\end{array} \right.
\eeq
Here $u = (u_1, \ldots, u_N)$ and each $u_i$ is a section of $(L_R L_K)^{r_i}$; $p$ is a section of $L_K^{-r}$; $A_K$ is a connection on $L_K$; the LG phase requires that $c_K< 0$. If the equation is unperturbed, then one can easily derives that solutions should satisfy $u = 0$ and $\ov\partial_A p = 0$, $F_{A_K} + {\bm i} (r |p|^2 + c_K) \sigma_c = 0$. This makes $(p, A_K)$ an abelian vortex, whose moduli space is identified with $rd$-fold symmetric product of $\Sigma_{\mc C}^*$, where $d = - {\rm deg} L_K$. In this way, the GLSM is closely related to the theory of Ross-Ruan \cite{Ross_Ruan} in the narrow case. Moreover, there is a way to perturb \eqref{eqn71} in the broad case in a fashion similar to the treatment of \cite{FJR3} to define the cohomological field theory in the broad case. This will be discussed in detail in \cite{Limit3}.

\subsection{The adiabatic limits}

In GLSM without a potential function, Gaio and Salamon \cite{Gaio_Salamon_2005} proved a correspondence between the gauged Gromov--Witten invariants of $X$ and the Gromov--Witten invariants of the symplectic reduction $\bar{X}$. Their method is to use the adiabatic limit of the vortex equation, namely, the $\epsilon \to 0$ limit of the equation
\begin{align*}
&\ \ov\partial_A u = 0,\ &\ * F_A + \epsilon^{-2} \mu(u) = 0.
\end{align*}
Gaio and Salamon essentially showed that in the limit, solutions to the above equation converges to holomorphic curves in $\bar{X}$ modulo bubbling. 

Similar phenomenon happens in the adiabatic limit of the gauged Witten equation in the geometric phase. Consider the $\epsilon$-dependent version of equation \eqref{eqn33}
\begin{align*}
&\ \ov\partial_A u + \nabla {\mc W} (u) = 0,\ &\ * F_{A_K} + \epsilon^{-2} \mu_K (u) = 0.
\end{align*}
One would like to prove that as $\epsilon \to 0$, solutions converges to holomorphic curves in $\bar{X}_W$ modulo bubbling, in a fashion similar to the case of Gaio--Salamon. We state the following conjecture, which will be studied in \cite{Limit2}.

\begin{conj}
There is a class $a \in H_{\rm CR}^*(\bar{X}_W; \Lambda)$ satisfying
\beqn
\lim_{\epsilon \to 0} \langle \alpha_1 \otimes \cdots \otimes \alpha_k; \beta \rangle_{g, k}^{GLSM;\epsilon} = \sum_{l \geq 0} \langle \alpha_1 \otimes \cdots \otimes \alpha_k \otimes \underbrace{a \otimes \cdots \otimes a}_{l}; \pi^* \beta\rangle_{g, k + l}^{GW}.
\eeqn
Here the left hand side is the GLSM correlation function with the parameter $\epsilon$ and the right hand side is the orbifold Gromov--Witten invariant of $\bar{X}_W$; $\pi: \ov{\mc M}_{g, k+l} \to \ov{\mc M}_k$ is the forgetful map. The class $a$ is defined by counting affine vortices whose images are contained in $X_W$, in a similar fashion as the quantum Kirwan map proposed by Salamon and studied by Ziltener \cite{Ziltener_book} and Woodward \cite{Woodward_1, Woodward_2, Woodward_3}.
\end{conj}

\subsection{The wall-crossing}

As formulated by Witten \cite{Witten_LGCY}, the GLSM can unify ``infrared'' theories in different geometric phases or geometric phases with Landau--Ginzburg phases in the ``ultraviolet.'' The Landau--Ginzburg/Calabi--Yau correspondence is therefore explained as a wall-crossing phenomenon. 

A general way to prove such correspondence, or more basically, a wall-crossing formula, is by constructing a ``master space.'' For simplicity, consider the symplectic vortex equation in a Hamiltonian $S^1$-manifold $(X, \omega, \mu)$. Suppose $\tau_0$ is a critical value of $\mu$ and $\tau_\pm = \tau_0 \pm \epsilon$. Then consider the union of moduli spaces of solutions to 
\beq\label{eqn72}
\ov\partial_A u = 0,\ * F_A + \mu(u)  - \tau = 0
\eeq
for all $\tau$ in $[\tau_-,\tau_+]$. Moreover, we should consider the moduli space of ``framed vortices,'' i.e., solutions $(P, A, u)$ to \eqref{eqn72} with a point in a fixed fibre $P_z$ of $P$. This gives a master space $\wt{\mc M}_{[\tau_-, \tau_+]}(X, \mu)$. Although for $\tau = \tau_0$, \eqref{eqn72} is not Fredholm, as a whole the master space has good deformation theory. Moreover, the value of $\tau$ is a moment map for a natural $S^1$-action on the master space. In principle, using the localization, one can derive a wall-crossing formula. The reader can refer to a construction of the master space \cite{BDW} and \cite{Gonzalez_Woodward_wall_crossing}.

The principle can be applied to the case of GLSM. In fact using this idea Chang--Li--Li--Liu \cite{CLLL_15, CLLL_16} derived an algorithm of computing Gromov--Witten invariants of quintic 3-folds. We hope to explore the wall-crossing phenomenon and carry out the master space construction in future study.

\bibliography{../../symplectic_ref,../../physics_ref}

\bibliographystyle{amsplain}

\end{document}